\documentclass{amsart}
\usepackage{setspace, amssymb, amsmath, amsfonts}

\theoremstyle{plain}
\newtheorem{thm}{Theorem}[section]
\newtheorem{theorem}[thm]{Theorem}

\newtheorem{lemma}[thm]{Lemma}
\newtheorem{prop}[thm]{Proposition}
\newtheorem{proposition}[thm]{Proposition}

\theoremstyle{definition}
\newtheorem{defn}[thm]{Definition}
\newtheorem{definition}[thm]{Definition}

\theoremstyle{remark}
\newtheorem{remark}[thm]{Remark}

\newcommand{\spect}{\mathcal{S}(2,k)}

\newcommand{\okdvt}{\mathcal{O}_{k, \cdot, v}^{\cdot, D, \cdot}(2)}
\newcommand{\okdvn}{\mathcal{O}_{k, \cdot, v}^{\cdot, D, \cdot}(n)}
\newcommand{\xo}{X_{O}}
\newcommand{\apK}{\mathfrak{d}_{\scriptscriptstyle{pK}}}

\newcommand{\fraka}{{\mathfrak a}}

\newcommand{\kp}{k}

\newcommand{\ccu}{(\wtu,\Gamma_U, \pi_U)}
\newcommand{\tp}{\tilde{p}}
\newcommand{\wtu}{\widetilde{U}}
\newcommand{\tv}{\tilde{v}}
\newcommand{\tw}{\tilde{w}}

\DeclareMathOperator{\vol}{Vol}

\begin{document}

\title[Spectral and geometric bounds on 2-orbifold diffeomorphism type]{Spectral and geometric bounds on 2-orbifold diffeomorphism type}

\author[E. Proctor]{Emily Proctor}
\author[E. Stanhope]{Elizabeth Stanhope}




\begin{abstract}
We show that a Laplace isospectral family of two dimensional Riemannian orbifolds, sharing a lower bound on sectional curvature, contains orbifolds of only a finite number of orbifold category diffeomorphism types.  We also show that orbifolds of only finitely many orbifold diffeomorphism types may arise in any collection of 2-orbifolds satisfying lower bounds on sectional curvature and volume, and an upper bound on diameter.  An argument converting spectral data to geometric bounds shows that the first result is a consequence of the second.
\end{abstract}

\maketitle

\begin{center}
\begin{small}
\noindent Emily Proctor \\ Department of Mathematics \\ Middlebury College \\ Middlebury, VT 05753

\noindent \texttt{eproctor@middlebury.edu} \rm

\bigskip

\noindent Elizabeth Stanhope \\ Department of
Mathematical Sciences \\ Lewis and Clark College \\ 0615 SW Palatine Hill Road, MSC 110 \\ Portland, OR 97219

\noindent \texttt{stanhope@lclark.edu} \rm
\end{small}
\end{center}

\vspace{5mm}

\section{Introduction}

The fact that the Laplace spectrum of a compact Riemannian manifold does not determine its geometry became popularly known with the 1992 announcement by Gordon, Webb and Wolpert \cite{GWW} that ``One cannot hear the shape of the drum."  Determining the information that the Laplace spectrum of a manifold \emph{does} contain about the geometry or topology of that manifold has been a productive endeavor for many decades.  One way to approach this question is to quantify the similarities shared by manifolds with a given spectrum.  For example, Osgood, Phillips and Sarnak \cite{OPS} showed that the spectrum of a closed surface determines the metric of the surface up to a family of metrics which are compact in the $C^{\infty}$-topology.  In the presence of a uniform lower bound on sectional curvature, Brooks, Perry and Petersen \cite{BPP} showed that isospectral sets of compact Riemannian manifolds, of dimension different than four, are finite up to diffeomorphism type.

In this paper we extend the result of Brooks, Perry and Petersen mentioned above to the category of two dimensional Riemannian orbifolds.  We prove the following theorem for compact, closed Riemannian 2-orbifolds:

\medskip

\noindent \bf{Main Theorem 1:} \it
For a fixed real number $k$, let $\spect$ denote the set of isospectral Riemannian 2-orbifolds with sectional curvature uniformly bounded below by $k$.  The collection $\spect$ contains orbifolds of only finitely many orbifold diffeomorphism types.
\rm

\medskip

\noindent In the process of obtaining this first theorem we prove a more general finiteness theorem with geometric bounds rather than spectral ones.  In particular we obtain the following result for compact, closed Riemannian 2-orbifolds:

\medskip

\noindent \bf{Main Theorem 2:} \it For $D>0$, $v >0$ and $k$ fixed real numbers, let $\okdvt$ denote the set of Riemannian 2-orbifolds with sectional curvature uniformly bounded below by $k$, diameter bounded above by $D$, and volume bounded below by $v$.  The collection $\okdvt$ contains orbifolds of only finitely many orbifold diffeomorphism types.
\rm

\medskip

\noindent The notation $\okdvt$ follows that in \cite{GP}.  By adding and deleting bounds, this notation can be used to express related statements.  For example, to formulate an orbifold analogue of the Cheeger finiteness theorem \cite{Ch} one would add an upper bound $K$ on sectional curvature and consider the set of $n$-orbifolds denoted $\mathcal{O}_{k, \cdot, v}^{K, D, \cdot}(n)$.

\medskip

An orbifold is a mild generalization of a manifold obtained by allowing coordinate patches to be modeled on $\mathbb{R}^n$ modulo the action of a finite group.  Orbifolds first arose as objects of study in algebraic geometry over a century ago.  Satake's \cite{S1} formulation of orbifolds in the language of differential geometry, under the name of $V$-manifold, appeared in 1956.  Later Thurston \cite{T} popularized $V$-manifolds among topologists and differential geometers under the name ``orbifolds."  Recently, interest in orbifolds has risen markedly due to their use in string theory (see \cite{ALR} for example).  This paper contributes to a new and expanding literature in the spectral geometry of Riemannian orbifolds.  For a concise survey of this literature the authors recommend the introduction to \cite{DGGW}.

We begin with a review of orbifold structures in Section \ref{background}.  In Section \ref{singularbounds} we show how the geometric bounds on $\okdvt$ imply finiteness for two aspects of the topology of an orbifold's singular set.  An application of Perelman's Stability Theorem in Section \ref{topology} shows that the underlying space of an orbifold in $\okdvt$ has one of only finitely many homeomorphism types.  These controls on the singular set and the underlying space of an orbifold in $\okdvt$ are combined to prove Main Theorem 2 in Section \ref{finiteoflddiffeo}.  Main Theorem 1 then follows from a short argument, recalled from \cite{St}, which shows that $\spect$ is a subset of $\okdvt$.

\subsection{Acknowledgments}  We thank Carolyn Gordon and David Webb for helpful discussions.  We also thank Peter Storm for a suggestion that clarified the proof of Proposition~\ref{okdvtprop2}.

\eject


\section{Orbifold background}\label{background}

In this section we detail the orbifold related definitions and notation that are used in what follows.  

\subsection{Definition of a Riemannian orbifold}

Just as a manifold is a topological space which locally has the structure of $\mathbb{R}^n$, an orbifold is a topological space which locally has the structure of $\mathbb{R}^n$ modulo the action of a finite group.   We have the following, which is a direct generalization of the definition of a manifold chart.

\begin{definition}
Let $\xo$ be a second countable Hausdorff topological space.  Given an open set $U$ contained in $\xo$, an \textit{orbifold chart} over $U$ is a triple $(\wtu, \Gamma_U, \pi_U)$ such that:
\begin{enumerate}
\item $\wtu$ is a connected open subset of $\mathbb{R}^n$,
\item\label{groupaction} $\Gamma_U$ is a finite group which acts on $\wtu$ by diffeomorphisms,
\item $\pi_U: \wtu \to U$ is a continuous map such that $\pi_U\circ\gamma = \pi_U$ for all $\gamma\in \Gamma_U$ and which induces a homeomorphism from $\wtu/\Gamma_U$ to $U$.
\end{enumerate} 
\end{definition}

As with manifolds, we cover the space $\xo$ with orbifold charts subject to a suitable compatibility condition (see page 2 in \cite{ALR}).  A smooth orbifold $O$ is the topological space $\xo$ together with a maximal atlas of orbifold charts.  The topological space $\xo$ is called the \textit{underlying space} of the orbifold.  

We note that if a group $\Gamma$ acts properly discontinuously on a manifold $M$, then the quotient space $M/\Gamma$ is an orbifold.  Any orbifold which can be realized as a quotient of a group action on a manifold in this way is called a \textit{good} orbifold.  Otherwise, the orbifold is called a \textit{bad} orbifold.

A Riemannian structure on an orbifold is defined by endowing the local cover $\wtu$ of each orbifold chart $\ccu$ with a $\Gamma_U$-invariant Riemannian metric.  By patching these local metrics together with a partition of unity we obtain a \emph{Riemannian orbifold}.  

\subsection{Local structures on Riemannian orbifolds}

Let $p$ be a point in an orbifold $O$ and take $\ccu$ an orbifold chart over a neighborhood $U$ of $p$.  If a point $\tp$ in $\pi_U^{-1}(p)$  has nontrivial isotropy, we say that $p$ is a \emph{singular} point.  The isomorphism class of the isotropy group of $\tp$ is independent of both choice of element of $\pi_U^{-1}(p)$ and choice of orbifold chart about $p$.  This isomorphism class is called the \emph{isotropy type} of $p$.  As in \cite{DGGW}, we call a chart about $p$ in a Riemannian orbifold a \emph{distinguished chart of radius $r$} if $\wtu$ is a convex geodesic ball of radius $r$ centered at point $\tp$ with $\pi_U(\tp) = p$.  In this situation the isotropy type of $p$ is the isomorphism class of the group coming from the chart, $\Gamma_U$.

We denote the tangent bundle of an orbifold $O$ by $TO$.  Here we shall simply recall the structure of a fiber of $TO$ over point $p\in O$, but refer the interested reader to \cite{S2} for more details.  Take $\ccu$ a distinguished chart about $p$ and let $\gamma \in \Gamma_U$.  The differential of $\gamma$ at $\tp$ acts on $T_{\tp}\wtu$.  Let $\Gamma_{U*\tp}$ denote the set of all such differentials.  The fiber of $TO$ over $p$, denoted $T_pO$, is defined to be $T_{\tp}\wtu/\Gamma_{U*\tp}$.  Fiber $T_pO$ is independent of choice of orbifold chart and is called the \emph{tangent cone to} $O$ \emph{at} $p$.  When $O$ is a Riemannian orbifold, the set of unit vectors in $T_pO$ is called the \emph{unit tangent cone to} $O$ \emph{at} $p$ and is denoted $S_pO$.

The tangent cone at a point in an orbifold need not be a vector space.  One consequence of this for Riemannian orbifolds is that the measure of the angle between vectors in a tangent cone needs a careful definition.  

\begin{defn}\label{angledef} Let $p$ be a point in a Riemannian orbifold that lies in an orbifold chart $\ccu$.  Take point $\tp \in \pi_U^{-1}(p)$.  Let $\pi_{U*\tp}$ denote the differential of $\pi_U$ at $\tp$.  For vectors $v$ and $w$ in $T_pO$, let $\tv_1, \tv_2, \dots, \tv_r$
  denote the elements of the set $(\pi_{U*\tp})^{-1}(v)$, and $\tw_1, \tw_2, \dots, \tw_s$
  denote the elements of the set $(\pi_{U*\tp})^{-1}(w)$.  The angle between $v$ and 
$w$ in $T_pO$ is defined to be
\begin{align*}
\angle(v,w) = \min_{\substack{ 
                       i = 1, 2 \dots, r \\
                       j = 1, 2 \dots, s 
                   }}   \{\angle(\tv_i, \tw_j)\}.
\end{align*}
\end{defn}

Finally, we are able to discuss curvature on a Riemannian orbifold by using local manifold covers.  We say that a Riemannian orbifold has sectional (resp. Ricci) curvature bounded below by $k$ if each point in the orbifold can be locally covered by a manifold with sectional (resp. Ricci) curvature bounded below by $k$. 

\subsection{Global structures on Riemannian orbifolds}\label{globalstructures}

We give a length space structure to a Riemannian orbifold using the distance function, 
\[d(p,q) = \inf\{\text{Length}(c)  :  c \ \text{is a continuous curve from} \ p \ \text{to} \ q\}.\]
When an orbifold $O$ is complete with respect to this metric, any two points in $O$ can be joined by a curve that achieves the distance between them.  Such a curve, parametrized with respect to arclength, is called a \emph{segment} in $O$.   Details on these ideas are given in \cite{Bz}.  

Smooth functions on an orbifold, as well as the Laplace operator acting on those functions, are described in \cite{Chi}.  For compact Riemannian orbifolds, by \cite{Chi} and \cite{DGGW}, the eigenvalue spectrum of the Laplace operator is a discrete set of positive real numbers, tending to infinity.  Two orbifolds are said to be \textit{isospectral} if they have the same eigenvalue spectrum.  

\begin{remark}\label{spectrumremark}
Chiang's original proof that the eigenvalue spectrum is a discrete set tending to infinity is based on Satake's original definition of $V$-manifold, for which the singular set has codimension at least 2.   An orbifold is a slight generalization of the notion of a $V$-manifold which has no restriction on the singular set.   Chiang's proof is extended in \cite{DGGW} to include all compact Riemannian orbifolds.  At the time that \cite{St} was published, only Chiang's result was known and hence many of the results in \cite{St} appear to depend on the codimension $\geq 2$ condition, though in fact, they do not.  We use some of these results from \cite{St} in Sections~\ref{singularbounds} and \ref{finiteoflddiffeo} below.  However, for this paper we state them using the most general definition of a Riemannian orbifold.
\end{remark}

\subsection{Smooth maps between orbifolds}

We now generalize the notion of a diffeomorphism of manifolds to the orbifold setting.  An orbifold diffeomorphism represents an equivalence of  the smooth orbifold structure as well as of the underlying topological space.  

The two definitions below come from \cite{ALR}.

\begin{definition}\label{weakmap}
Let $O_1$ and $O_2$ be orbifolds.  A \textit{smooth orbifold map} $f:O_1\to O_2$ consists of a continuous map from $X_{O_1}$ to $X_{O_2}$ such that for any $x\in O_1$ there are orbifold charts $(\widetilde U, \Gamma_U, \pi_U)$ over neighborhood $U$ of $x$ and $(\widetilde V, \Gamma_V, \pi_V)$ over neighborhood $V$ of $f(x)$ such that:
\begin{enumerate}
\item $f(U) \subset V$,
\item there exists a smooth lift $\tilde f$ of $f$ carrying $\widetilde U$ to $\widetilde V$ for which $\pi_V \circ \tilde f = f \circ \pi_U$.
\end{enumerate}
\end{definition}

\begin{definition}\label{diffeo}
Orbifolds $O_1$ and $O_2$ are diffeomorphic if there exist smooth orbifold maps $f:O_1\to O_2$ and $g:O_2\to O_1$ such that $f \circ g = 1_{O_2}$ and $g \circ f = 1_{O_1}$.
\end{definition}

\begin{remark}  The literature contains several different definitions of maps between orbifolds.  Although maps given by Definition~\ref{weakmap} are weak in the sense that they behave poorly with respect to bundles, they suffice for the present discussion.  In particular, effective orbifolds which are diffeomorphic via Definition~\ref{diffeo} have strongly diffeomorphic groupoid presentations. 
\end{remark}

\vspace{15mm}


\section{Bounds on the singular set of a 2-orbifold}\label{singularbounds}

Singular points in a smooth 2-orbifold have one of only three possible forms.  A cone point is locally modeled on a disk in the plane modulo the action of a cyclic group of rotations, a mirror point is modeled on a disk modulo a reflection, and a dihedral point is modeled on a disk modulo the action of a dihedral group. 

There are also only three forms that a connected component of the singular set of a compact 2-orbifold without boundary can take.  If the underlying space of the 2-orbifold is a surface without boundary, the only singular points are isolated cone points.  When the underlying space has non-empty boundary, each component of its boundary is circular.  These circles do not form an orbifold boundary, however, because they are either reflector circles, made up entirely of mirror points, or reflector crowns, which consist of a finite number of dihedral points linked together by continua of mirror points.  Thus, when the underlying space of a compact 2-orbifold has boundary, the connected components of its singular set are some combination of cone points, mirror circles and reflector crowns.  For further details, see Section 13.3 in \cite{T}.

In this section we establish a universal upper bound on the number of connected components of the singular set of any orbifold in $\okdvt$.  In addition we prove that the number of dihedral points in an orbifold in $\okdvt$ is universally bounded above.  These controls, combined results from Section \ref{topology} and \cite{St}, will be used in Section \ref{finiteoflddiffeo} to prove the two main theorems of this paper.

We begin with two technical lemmas.  Suppose that $K$ is a compact subset of complete Riemannian orbifold $O$.  For $p \in O$, let $\apK \subset S_pO$ denote the set of initial velocity vectors of segments running from $p$ to $K$.  We call $\apK$ the set of directions from $p$ to $K$.  Also, given a subset $\fraka$ of the unit $n$-sphere $S^n$, we define
\begin{align*}
\fraka(\theta) &= \{v \in S^n : \angle(\fraka,v) < \theta\}. 
\end{align*}

\begin{lemma}\label{technicallemma1} Let $O \in \okdvn$ and $p, q
  \in O$.   Then there exist $\alpha \in (0,
  \frac{\pi}{2})$ and $r>0$ such that if
\begin{align*}
\frak{d}_{pq}(\tfrac{\pi}{2} + \alpha) =
S_pO, \ \text{and} \ \frak{d}_{qp}(\tfrac{\pi}{2} + \alpha) = S_qO,
\end{align*}
then $d(p, q) \ge r$.  The constants $\alpha$ and $r$ depend only on $k$,
$D$, $v$ and $n$.
\end{lemma}

\begin{proof}  The statement of this lemma is precisely that of Lemma 8.2 in \cite{St}, using the more general definition of an orbifold (see Remark~\ref{spectrumremark}) and without the requirement that $O$ have only isolated singularities.  The assumption about isolated singularities is never used in the proof given in \cite{St}, so the result holds for orbifolds with general singularities as well.
\end{proof}

The second technical lemma shows that, for $r>0$, there is a universal upper bound on the size of a set of pairwise $\ge r$-apart points in $O \in \okdvn$.  We recall that a \emph{minimal $\varepsilon$-net} is an ordered set of points $p_1, \dots, p_N$ in a metric space such that the open balls $B(p_i,\varepsilon)$ cover the metric space, but the open balls $B(p_i, \tfrac{\varepsilon}{2})$ are pairwise disjoint.  When the metric space is compact and connected, it is known that one can find a minimal $\varepsilon$-net in that space for any $\varepsilon>0$.

\begin{lemma}\label{technicallemma2}  Suppose that $O$ is an $n$-dimensional orbifold with diameter bounded above by $D>0$ and Ricci curvature greater than or equal to $(n-1)k$.  Also suppose that $\{p_1, p_2, \dots, p_m\}$ is a set of points in $O$ for which $d(p_i,p_j) \ge r>0$ with $i,j = 1, 2, \dots, m$, $i \ne j$.  Then there is a constant $C(r, k, D, n)$ such that $m \le C$.
\end{lemma}

\begin{proof}  Let $\{x_1, x_2, \dots, x_N\}$ be a minimal $\frac{r}{2}$-net in $O$.  Without loss of generality assume $B(x_1, \tfrac{r}{4})$ has the minimal volume among all of the $\frac{r}{4}$-balls about points in this net.  Because the $\frac{r}{4}$-balls are disjoint,
\[N \vol B(x_1, \tfrac{r}{4}) \le \sum_{i=1}^N \vol B(x_i, \tfrac{r}{4}) \le \vol O.\]
Thus $\vol B(x_1, \frac{r}{4}) \le \vol O/N$.  

Recall for $p \in O$ and $0 \le s \le S$, the Relative Volume Comparison Theorem for orbifolds in \cite{Bz} implies,
\begin{align}\label{inequality}
\frac{\vol B(p, S)}{\text{Vol}B(p, s)} \le \frac{\vol B^n_{\kp}(S)}{\text{Vol}B^n_{\kp}(s)}
\end{align}  
where $B^n_{\kp}(r)$ denotes the geodesic ball of radius $r$ in the simply connected $n$-dimensional space form of constant curvature $k$.  Now apply line~ (\ref{inequality}) with $p=x_1$, $s= \frac{r}{4}$ and $S=D$.  This yields
\begin{align}\label{volfact2}
\frac{\vol B(x_1, D)}{\text{Vol}B(x_1, \frac{r}{4})} \le \frac{\vol B^n_{\kp}(D)}{\text{Vol}B^n_{\kp}(\frac{r}{4})}.
\end{align}
Using $\vol B(x_1,D) = \vol O$ and $\vol B(x_1, r/4) \le \vol O/N$ we find that line (\ref{volfact2}) becomes
\begin{align*}
N \le \frac{\vol B^n_{\kp}(D)}{\text{Vol}B^n_{\kp}(\frac{r}{4})} = C(r, k, D, n),
\end{align*}
yielding a universal upper bound on the number of elements in the minimal $\frac{r}{2}$-net.

Because each pair of points in $\{p_1, p_2, \dots, p_m\}$ are at least a distance $r$ apart from each other, there can be at most one of these points per open $\frac{r}{2}$-ball.  Thus the bound on the number of elements in our minimal $\frac{r}{2}$-net  is also a bound on $m$.  In particular we have $m \le N \le C(r, k,D, v) $.
\end{proof}

\begin{prop}\label{components}  There is a universal upper bound on the number of connected components of the singular set of an orbifold in $\okdvt$.
\end{prop}
\begin{proof} 

Let $O\in\okdvt$.  The desired upper bound is the sum of upper bounds $B_C$ on the number of cone points in $O$, $B_R$ on the number of reflector crowns in $O$, and $B_M$ on the number of mirror circles in $O$.  We derive each of these bounds and confirm that each depends only on $k, D$ and $v$.  

The proof of the existence of the upper bound $B_C$ is the same as the proof of a similar bound in Proposition 8.3 of \cite{St}.  The bound is achieved by showing that if $p$ and $q$ are any two cone points in an orbifold in $\okdvt$, then since the isotropy groups of $p$ and $q$ are cyclic, it must be the case that there exists $\alpha \in (0,
  \frac{\pi}{2})$ such that $\frak{d}_{pq}(\tfrac{\pi}{2} + \alpha) =
S_pO \ \text{and} \ \frak{d}_{qp}(\tfrac{\pi}{2} + \alpha) = S_qO$.   By Lemma~\ref{technicallemma1} we conclude that $p$ and $q$ are at least a distance $r$ apart from each other.  An application of Lemma~\ref{technicallemma2} implies that there is a universal constant $B_C$ bounding the number of cone points in an orbifold in $\okdvt$.

Essentially the same argument can be used to obtain the bound $B_R$ on the number of reflector crowns.  In this case, we know that there must be at least one dihedral point per crown so we show there is a universal bound, $B_D$, on the number of dihedral points in an orbifold in $\okdvt$.  We first note that the isotropy group $\Gamma$ of any dihedral point contains a cyclic subgroup and then follow the argument for the cone point case.

Obtaining the bound $B_M$ on the number of mirror circles in $O$ is a bit more subtle.  To begin, list the mirror circles in $O$ as $S_1, S_2, \dots, S_m$.  For $1 \le j \le m$, let $q_j$ be a point on mirror circle $S_j$.  Mirror circles in $O$ are non-intersecting since any point of intersection would fail to have one of the three singular structures mentioned above.  Thus the set $Q = \{q_1, q_2, \dots, q_m\}$ contains $m$ distinct points.  As in the other two cases, we will show that any two points in $Q$ are a distance greater than or equal to $r$ apart and apply Lemma \ref{technicallemma2} to obtain the required bound.  

Suppose $S_i$ and $S_j$ are distinct mirror circles in $O$.  The distance between these circles, $d(S_i, S_j)$, is the infimum of the distance function (see Section~\ref{globalstructures})
restricted to $S_i \times S_j$.  We will show that there is an $r>0$ for which $d(S_i, S_j)>r$, and in so doing conclude that points in set $Q$ are pairwise greater than or equal to $r$ apart.  Begin by observing that because $S_i \times S_j$ is compact, we can take $(p_i,p_j) \in S_i \times S_j$ an ordered pair for which $d(p_i,p_j)=d(S_i, S_j)$.  Also, because $S_i$ and $S_j$ are nonintersecting closed sets, $d(S_i, S_j)>0$.  Let $\gamma$ be the segment of length $d(p_i,p_j)$ running from $p_i$ to $p_j$.  

We claim that the segment $\gamma$ is perpendicular to mirror circle $S_i$ at $p_i$ and to mirror circle $S_j$ at $p_j$.  Take a distinguished coordinate chart $\ccu$ for neighborhood $U$ about $p_i$ such that $\wtu$ is an open neighborhood of the origin in $\mathbb{R}^2$, $\pi_U(0) = p_i$, and $\Gamma_U = \{id, \rho\}$ where $id$ is the identity map on $\wtu$ and $\rho$ is reflection of $\wtu$ across the $y$-axis.  By Proposition 15 in \cite{Bz}, there exists a point $z$ in $U$ that lies on $\gamma$ but does not lie on the mirror edge of $U$.  This means $\pi_U^{-1}(z) = \{\tilde{z}_1, \tilde{z}_2\}$ with $\tilde{z}_1$ and $\tilde{z}_2$ distinct in $\wtu$ and $\rho(\tilde{z}_1) = \tilde{z}_2$.  Let $\tilde{\gamma}_1$ be the lift of $\gamma$ that connects $0$ and $\tilde{z}_1$.  Because $\gamma$ minimizes the distance from $S_i$ to $S_j$, segment $\tilde{\gamma}_1$ achieves the minimal distance from $\tilde{z}_1$ to the $y$-axis in $\wtu$.  Observing that the $y$-axis is a submanifold of $\wtu$, we see that $\tilde{\gamma}_1$ is orthogonal to the $y$-axis at $0$.  We conclude that $\gamma$ is perpendicular to mirror circle $S_i$ at $p_i$.  The argument for $p_j$ on $S_j$ is similar.

Now $\frak{d}_{p_i p_j}\subset S_{p_i}O$ contains the initial tangent vector to $\gamma$.  Because $\gamma$ is perpendicular to the mirror circle $S_i$ at $p_i$, this initial vector lifts to two antipodal vectors in $S_{0}\wtu$.  Thus no matter the value of $\alpha>0$ from Lemma~\ref{technicallemma1}, the fact that $\alpha$ is nonzero implies $\frak{d}_{p_i p_j}(\frac{\pi}{2}+\alpha)=S_{p_i}O$.  Similarly $\frak{d}_{p_j p_i}(\frac{\pi}{2}+\alpha)=S_{p_j}O$.  Having satisfied the hypotheses of Lemma~\ref{technicallemma1}, we obtain $r>0$ such that $d(S_i,S_j)=d(p_i,p_j) \ge r$.

We have shown that points in set $Q$ are pairwise greater than or equal to $r$ apart.  Lemma~\ref{technicallemma2} provides universal bound $B_M$ on the number of points in $Q$, and thus on the number of mirror circles in $O$.
\end{proof}

Note that the universal constant $B_D$ obtained in the proof above is the one required to prove the following proposition.

\begin{proposition}\label{dihedralpoints}  The number of dihedral points in an 2-orbifold in $\okdvt$ is universally bounded above.
\end{proposition}

\vspace{15mm}


\section{Controls on underlying space topology}\label{topology}

Using Perelman's Stability Theorem for Alexandrov spaces, we observe that the underlying space of an orbifold in $\okdvn$ has one of only a finite number of homeomorphism types.  We begin by recalling Perelman's theorem, denoting the Gromov-Hausdorff metric by $d_{GH}$.  Perelman \cite{P} originally proved this result in 1991, but it remained in preprint form.  Kapovitch \cite{K} ultimately wrote the result for wider distribution.

\begin{theorem}\label{stability}  Suppose $X$ is an $n$-dimensional Alexandrov space with curvature bounded below by $k$.  Then there exists $\varepsilon = \varepsilon(X)>0$ such that if $Y$ is an $n$-dimensional Alexandrov space and $d_{GH}(X,Y)<\varepsilon$, then $Y$ is homeomorphic to $X$.
\end{theorem}

This theorem applies in our context because an orbifold with sectional curvature bounded below by $k$ is an example of an Alexandrov space with curvature bounded below by $k.$   We obtain our underlying space topological finiteness result with the following lemma and a compactness argument.

\begin{lemma}\label{precompact}  The set $\okdvn$ is precompact relative to the Gromov-Hausdorff metric.  Limit points of this set are Alexandrov spaces of dimension $n$ with the same lower bound on curvature and upper bound on diameter.
\end{lemma}

\begin{proof}  The precompactness result follows from Gromov's Compactness Theorem (Theorem 10.7.2 in \cite{BuBuI}) and relies on the universal upper diameter bound $D$ and the universal constant given in Lemma~\ref{technicallemma2}.  Theorem 10.7.2 in \cite{BuBuI} implies that the limit points are Alexandrov spaces with curvatures bounded below by $k$ and diameters bounded above by $D$.  By Corollary 10.10.11 in \cite{BuBuI}, the lower volume bound on $\okdvn$ prevents collapsing, so the dimension of any limit space is $n$.
\end{proof}

\begin{prop}\label{ustoptype}  The underlying space of an orbifold in $\okdvn$ has one of only a finite number of homeomorphism types.
\end{prop}

\begin{proof}  Arguing by contradiction, we suppose that $\{O_i\}$ is an infinite sequence  of orbifolds in $\okdvn$ each having an underlying space of a distinct homeomorphism type.  Lemma~\ref{precompact} implies that this sequence has a $d_{GH}$-convergent subsequence, and that the limit space $X$ is an $n$-dimensional Alexandrov space with curvature bounded below by $k$.  If we choose $\varepsilon = \varepsilon(X)$ as in Theorem~\ref{stability}, there exists $N$ such that $d_{GH}(O_i, X)<\varepsilon$ for all $i>N$.  But then each $O_i$, with $i>N$, must be homeomorphic to $X$. This is a contradiction.
\end{proof}


\section{Finiteness of orbifold diffeomorphism types}\label{finiteoflddiffeo}

We use the finiteness results from Propositions \ref{components}, \ref{dihedralpoints} and \ref{ustoptype}, as well as a result from \cite{St}, to prove orbifold diffeomorphism finiteness for orbifolds in $\okdvt$.  Once this result is established, a brief argument shows that orbifolds in $\spect$ have only finitely many orbifold diffeomorphism types.

For ease of exposition, we break $\okdvt$ into four disjoint subsets following Theorem 13.3.6 in \cite{T}.  In particular we write
\[\okdvt = \mathcal{B} \sqcup \mathcal{E} \sqcup \mathcal{P} \sqcup \mathcal{H}\]
where $\mathcal{B}$ are the bad orbifolds, $\mathcal{E}$ the elliptic orbifolds, $\mathcal{P}$ the parabolic orbifolds, and $\mathcal{H}$ the hyperbolic orbifolds.

Our argument begins with a lemma from which orbifold diffeomorphism finiteness for non-hyperbolic 2-orbifolds follows immediately.

\begin{lemma}\label{isotropy}  Orbifolds in $\okdvt$ contain points of only finitely many possible isotropy types.
\end{lemma}
\begin{proof}  This follows directly from Main Theorem 1 in \cite{St} and Remark~\ref{spectrumremark}.
\end{proof}

\begin{prop}\label{okdvtprop1}
The subset $\mathcal{B}\sqcup \mathcal{E} \sqcup \mathcal{P} \subset \okdvt$ contains orbifolds of only finitely many orbifold diffeomorphism types.  
\end{prop}

\begin{proof}  Suppose $\mathcal{B}\sqcup \mathcal{E} \sqcup \mathcal{P}$ contains orbifolds of infinitely many orbifold diffeomorphism types.  By Thurston's classification of bad, elliptic and parabolic 2-orbifolds, this implies orbifolds in this collection contain points of arbitrarily large order isotropy type.  This contradicts Lemma~\ref{isotropy}.
\end{proof}

We next consider the case of the hyperbolic orbifolds $\mathcal{H} \subset \okdvt$.  In what follows, we say that two reflector crowns \emph{have the same type} if they have the same number of dihedral points and if, when listed in order, the isotropy types of the dihedral points in the first reflector crown match the isotropy types of the dihedral points in the second reflector crown, up to a cyclic permutation.

\begin{prop}\label{okdvtprop2}
The subset $\mathcal{H} \subset \okdvt$ contains orbifolds of only finitely many orbifold diffeomorpism types.  
\end{prop}

\begin{proof}
Partition $\mathcal{H}$ so that within each partition element orbifolds have the same number and type of cone points, number of reflector circles, number and type of reflector crowns, and underlying space homeomorphism type.  Together Propositions~\ref{components}, \ref{dihedralpoints}, \ref{ustoptype}, and Lemma~\ref{isotropy} imply that this partition has a finite number of elements.

We prove that $\mathcal{H}$ contains orbifolds of only finitely many orbifold diffeomorphism types by showing that pairs of orbifolds within a partition element must be orbifold diffeomorphic.  To begin, let $O_1$ and $O_2$ be orbifolds in the same element of the partition on $\mathcal{H}$.  Because $O_1$ and $O_2$ are smooth, Thurston's orbifold classification implies that each of these orbifolds is orbifold diffeomorphic to a quotient of the hyperbolic plane by the properly discontinuous action of a group of isometries. Denote these quotient structures by $O_1 = \mathbb{H}^2/\Gamma_1$ and $O_2 = \mathbb{H}^2/\Gamma_2$.   

In \cite{M} it is shown that because $O_1$ and $O_2$ have underlying spaces with the same genus and orientability, and they have matched singular data, there exists a group isomorphism $\varphi:\Gamma_1\to\Gamma_2$.  Because $\Gamma_1$ and $\Gamma_2$ are cocompact, an application of Theorem 8.16 in \cite{Ka}  yields a $\varphi$-equivariant quasi-M\"obius homeomorphism $f$ from the boundary circle of $\mathbb{H}^2$ at infinity to itself.  Let $\tilde f:\mathbb{H}^2 \rightarrow \mathbb{H}^2$ denote the Douady-Earle extension of $f$ (see Section 8.4 in \cite{Ka}).  The properties of the Douady-Earle extension imply that $\tilde f$ is a $\varphi$-equivariant diffeomorphism.

Because the diffeomorphism $\tilde f$ is $\varphi$-equivariant, it induces a homeomorphism $h$ of the underlying spaces of $O_1$ and $O_2$.
Using the global orbifold charts on $O_1$ and $O_2$ provided by their quotient structures, the map $\tilde f$ is precisely what is needed to conclude $h$ is a diffeomorphism of orbifolds.  Therefore $O_1$ and $O_2$ are orbifold diffeomorphic.
\end{proof}

We are now in a position to prove the two Main Theorems.

\medskip

\noindent \bf{Main Theorem 2:} \it For $D>0$, $v >0$ and $k$ fixed real numbers, let $\okdvt$ denote the set of Riemannian 2-orbifolds with sectional curvature uniformly bounded below by $k$, diameter bounded above by $D$, and volume bounded below by $v$.  The collection $\okdvt$ contains orbifolds of only finitely many orbifold diffeomorphism types.
\rm

\begin{proof}  Apply Propositions~\ref{okdvtprop1} and \ref{okdvtprop2}.
\end{proof}

\medskip

\noindent \bf{Main Theorem 1:} \it
For a fixed real number $k$, let $\spect$ denote the set of isospectral Riemannian 2-orbifolds with sectional curvature uniformly bounded below by $k$.  The collection $\spect$ contains orbifolds of only finitely many orbifold diffeomorphism types.
\rm

\begin{proof}
By the orbifold version of Weyl's asymptotic formula \cite{F}, we know that all orbifolds in $\spect$ must have the same volume $v$.  In addition Proposition 7.4 in \cite{St} states that a collection of isospectral orbifolds, satisfying a uniform lower bound $k(n-1)$ on Ricci curvature, has a corresponding upper bound $D$ on diameter.  Thus $\spect$ is a subcollection of $\okdvt$ and we apply Main Theorem 2.
\end{proof}


\end{document}